\documentclass[reqno,11pt]{amsart}
\usepackage{amsmath, amsfonts,amsthm,amssymb,amsbsy,upref,color,graphicx,amscd,hyperref,enumerate,xfrac}
\usepackage[active]{srcltx}
\usepackage[latin1]{inputenc}
\usepackage{bbm}
\usepackage{algorithm}
\usepackage{algorithmicx}
\usepackage{algpseudocode}
\usepackage{graphicx}
\usepackage{subfig} 
\usepackage{graphicx}
\usepackage[bb=boondox]{mathalfa}

\usepackage{pgf}
\usepackage{tikz}
\usetikzlibrary{patterns}
\usepackage{sansmath}
\usetikzlibrary{shadings,intersections,patterns}
\usetikzlibrary{intersections, calc, angles}
\usepackage{tkz-euclide}
\usepackage{pgfplots}
\pgfplotsset{compat=1.10}
\usepgfplotslibrary{fillbetween}
\usepackage{xcolor,xfrac}
\definecolor{racing}{rgb}{0.7,0.1,0.2}
\definecolor{french}{rgb}{0,0.2,0.7}
\usepackage{verbatim}

\usetikzlibrary{arrows,automata}
\usetikzlibrary{positioning}
\usetikzlibrary{calc,through}

\tikzset{
	state/.style={
		rectangle,
		rounded corners,
		draw=black, very thick,
		minimum height=2em,
		inner sep=2pt,
		text centered,
	},
}

\def\ds{\displaystyle}
\def\eps{{\varepsilon}}
\def\N{\mathbb{N}}

\def\R{\mathbb{R}}

\def\HH{\mathcal{H}}

\newcommand{\be}{\begin{equation}}
\newcommand{\ee}{\end{equation}}
\newcommand{\bib}[4]{\bibitem{#1}{\sc#2: }{\it#3. }{#4.}}

\newcommand{\cp}{\mathop{\rm cap}\nolimits}

\numberwithin{equation}{section}
\theoremstyle{plain}
\newtheorem{teo}{Theorem}[section]
\newtheorem{lemma}[teo]{Lemma}

\newtheorem{prop}[teo]{Proposition}

\theoremstyle{remark}
\newtheorem{oss}[teo]{Remark}

\newcommand{\mean}[1]{\,-\hskip-1.08em\int_{#1}} 

\theoremstyle{plain}

\newcommand{\mathbbmm}[1]{\text{\usefont{U}{bbm}{m}{n}#1}}
\newcommand{\ind}{\mathbbmm{1}}

\title{A two\,-\,phase problem with Robin conditions on the free boundary}

\makeindex
\author[S. Guarino Lo Bianco, D. A. La Manna, B. Velichkov]{
	Serena Guarino Lo Bianco, Domenico Angelo La Manna, Bozhidar Velichkov}
\address{
	Serena Guarino Lo Bianco\\
	Universit\`a degli studi di Napoli ``Federico II''\\
	Dipartimento di Agraria\\
	Via Universit\`a 100,
	80055 Portici (NA), Italia.
}
\email{ serena.guarinolobianco@unina.it}

\address{
Domenico Angelo La Manna\\
 University of Jyvaskyl\"a, Department of Mathematics and Statistics,  P.O. Box 35 (MaD)
FI-40014, Finland}
\email{domenicolamanna@hotmail.it}
\address {Bozhidar Velichkov: \newline \indent
	Dipartimento di Matematica, Universit\`a di Pisa \newline \indent
	Largo Bruno Pontecorvo, 5, 56127 Pisa - ITALY}
\email{bozhidar.velichkov@unipi.it}

\begin{document}

\begin{abstract}
We study for the first time a two-phase free boundary problem in which the solution satisfies a Robin boundary condition. We consider the case in which the solution is continuous across the free boundary and we prove an existence and a regularity result for minimizers of the associated variational problem. Finally, in the appendix, we give an example of a class of Steiner symmetric minimizers.
\end{abstract}


\maketitle


\section{Introduction}\label{s:intro}
For a fixed a constant $\beta>0$ and a smooth bounded open set $D\subset \R^d$, $d\ge 2$, we consider the functional 
$$J_\beta(u,\Omega)=\int_D|\nabla u|^2\,dx+\beta\int_{\partial^\ast \Omega}u^2\,d\HH^{d-1},$$
defined on the couples $(u,\Omega)$, where $u\in H^1(D)$, $\Omega\subset\R^d$ is a set of finite perimeter in the sense of De Giorgi (see Section \ref{s:prelimiaries})
 and $\partial^\ast\Omega$ denotes the reduced boundary of $\Omega$ (see Section \ref{s:prelimiaries}). Recall that, when $\Omega$ is smooth, $\partial^\ast\Omega$ is the topological boundary of $\Omega$.
\medskip

In this paper we study the existence and the regularity of minimizers of the functional $J_\beta$ 
among all couples $(u,\Omega)$, which are fixed outside the domain $D$. Precisely, throughout the paper, we fix a set $E\subset\R^d$ of finite perimeter, a constants $m>0$ and a function 
$$v\in H^1_{loc}(\R^d)\quad\text{such that}\quad   v\ge m\quad\text{in}\quad\R^d\quad\text{and}\quad\int_{\partial^\ast E}v^2\,d\HH^{d-1}<+\infty\ ;$$ 
we define the admissible sets
$$\mathcal V=\big\{u\in H^1_{loc}(\R^d)\ :\ u-v\in H^1_0(D)\big\},$$
$$\mathcal E=\big\{\Omega\subset\R^d\ :\ \text{\rm Per}(\Omega)<+\infty\ \text{ and }\ \Omega=E\ \text{ in }\ \R^d\setminus D\big\},$$
and we consider the variational minimization problem 
\begin{equation}\label{e:pb}
\min\big\{J_\beta(u,\Omega)\ :\ u\in\mathcal V,\ \Omega\in\mathcal E\big\}.
\end{equation}
Our main result is the following.
%
%
%
\begin{teo}[Existence and regularity of minimizers]\label{t:main}
Let $\beta>0$, $D\subset\R^d$, $v$, $E$, $\mathcal V$ and $\mathcal E$ be as above. Then the following holds.

\begin{enumerate}[\rm (i)]
\item There exists a solution $(u,\Omega)\in\mathcal V\times\mathcal E$ to the variational problem \eqref{e:pb}.
\item For every solution $(u,\Omega)$ of \eqref{e:pb}, $u$ is H\"older continuous and bounded from below by a strictly positive constant in $D$. 

\item If $(u,\Omega)$ is a solution to \eqref{e:pb}, then the free boundary $\partial\Omega\cap D$ can be decomposed as the disjoint union of a regular part $\text{\rm Reg}(\partial\Omega)$ and a singular part $\text{\rm Sing}(\partial\Omega)$, where :
\begin{enumerate}[\quad $\bullet$]
	\item $\text{\rm Reg}(\partial\Omega)$  is a $C^{\infty}$ hypersurface and a relatively open subset of $\partial\Omega$, and the function $u$ is $C^\infty$ smooth on $\text{\rm Reg}(\partial\Omega)$;
	\item $\text{\rm Sing}(\partial\Omega)$ is a closed set, which is empty if $d\le 7$, discrete if $d=8$, and of Hausdorff dimension $d-8$, if $d>8$.
\end{enumerate}		
%
%

\end{enumerate}
\end{teo}	
 \begin{oss}
We notice that if $(u,\Omega)$ is a solution to \eqref{e:pb}, then $u$ is harmonic in the interior of $\Omega$ and $D\setminus\Omega$. Thus, as a consequence of Theorem \ref{t:main} (iii), in a neighborhood of a regular point $x_0\in\text{\rm Reg}(\partial\Omega)$, the functions $u:\overline\Omega\to\R$ and $u:\overline{D\setminus\Omega}\to\R$ are $C^\infty$ up to the free boundary $\partial\Omega$.
 \end{oss}	
\subsection{Outline of the proof and organization of the paper}\label{sub:sketch}
The main difficulty in the proof of Theorem \ref{t:main} is to prove the existence of a minimizing couple $(u,\Omega)$ and to show that the function $u$ is H\"older continuous and bounded from below by a strictly positive constant in $D$. The almost-minimality of the solutions is proved in Theorem \ref{t:regularity}. Finally, in the Appendix, we give examples of minimizers in domains $D$ symmetric with respect to the hyperplane $\{x_d=0\}$.


\subsubsection{Existence}The existence of a solution $(u,\Omega)$ and the regularity of $u$ (H\"older regularity and non-degeneracy) are treated simulatenously. The reason is that if $(u_n,\Omega_n)$ is a minimizing sequence for \eqref{e:pb}, then in order to get the compactness of $\Omega_n$, we need a uniform bound (from above) on the perimeter $\text{Per}(\Omega_n)$, for which we need the functions $u_n$ to be bounded from below by a strictly positive constant. Now, notice that we cannot simply replace $u_n$ by $u_n\vee \eps$, for some $\eps>0$; this is due to the fact that the second term in $J_\beta$ is increasing in $u$ :  $$\ds\int_{\partial^\ast\Omega_n}u_n^2\,d\HH^{d-1}\le \int_{\partial^\ast\Omega_n}(\eps\vee u_n)^2\,d\HH^{d-1}.$$ 
Thus, we select a minimizing sequence which is in some sense optimal. Precisely, we take $(u_n,\Omega_n)$ to be solution of the auxiliary problem
\begin{equation}\label{e:pb_approx_intro}
\min\Big\{J_\beta(u,\Omega)\ :\ u\in\mathcal V,\ \Omega\in\mathcal E,\ u\ge \frac1n\ \text{in}\ D\Big\}, 
\end{equation}
for which the existence of an optimal set is much easier (see Section \ref{s:approx}, Proposition \ref{p:existence:aux}). Still, we do not have a uniform (independent from $n$) bound from below for the functions $u_n$, so we still miss the uniform bound on the perimeter of $\Omega_n$.

On the other hand, we are able to prove that the sequence $u_n$ is uniformly H\"older continuous in $D$ (see Section \ref{s:approx}, Lemma \ref{p:holder}). This allows to extract a subsequence $u_n$ that  converges locally uniformly in $D$ to a non-negative H\"older continuous function $u_\infty:D\to\R$ (see Section \ref{s:existence}). Now, on each of the sets $\{u_\infty>t\}$, $t>0$, the sequence $\Omega_n$ has uniformly bounded perimeter. This allows to extract a subsequence $\Omega_n$ that converges pointwise almost-everywhere on $\{u_\infty>0\}$ to some $\Omega_\infty$. Thus, we have constructed our candidate for a solution: $(u_\infty,\Omega_\infty)$.

In order to prove that $(u_\infty,\Omega_\infty)$ is an admissible competitor in \eqref{e:pb}, we need to show that $\Omega_\infty$ has finite perimeter. We do this in Section \ref{s:existence}.
We first use the optimality of $(u_n,\Omega_n)$ to prove that $(u_\infty,\Omega_\infty)$ is optimal when compared to a special class of competitors. This optimality condition can be written as (we refer to  Lemma \ref{l:opt_cond} for the precise statment) :
\begin{equation}\label{e:fake_opt_cond}
J_\beta(u_\infty,\Omega_\infty)\le J_\beta(u_t,\Omega_t)\quad\text{where}\quad u_t=u_\infty\vee t\quad\text{and}\quad \Omega_t=\Omega_\infty\cup \{u_\infty\le t\},
\end{equation}
for any $t>0$. Next, from this special optimality condition we deduce that the function $u_\infty$ is bounded from below by a strictly positive constant (see Proposition \ref{p:non-degeneracy}). From this, in Section \ref{s:existence},  we deduce that $\Omega_\infty$ has finite perimeter in $\R^d$ and that the couple $(u_\infty,\Omega_\infty)$ is a solution to \eqref{e:pb}.


\subsubsection{H\"older continuity and non-degeneracy of $u$} Let now $(u,\Omega)$ be any solution of \eqref{e:pb}. In order to prove the H\"older continuity and the non-degeneracy of $u$ it is sufficient to exploit some of the estimates that we already used to prove the existence. Indeed, we can test the optimality of $(u,\Omega)$ with the competitors from \eqref{e:fake_opt_cond}. Thus, for $t>0$ small enough, we have 
\begin{equation}\label{e:fake_opt_cond2}
J_\beta(u,\Omega)\le J_\beta(u_t,\Omega_t)\quad\text{where}\quad u_t=u\vee t\quad\text{and}\quad \Omega_t=\Omega\cup \{u\le t\}.
\end{equation}
In particular, 
\begin{align}
\int_{D}|\nabla u|^2\,dx+\beta\int_{\partial^\ast\Omega}u^2\notag&\le \int_{D}|\nabla (u\vee t)|^2\,dx+\beta\int_{\partial^\ast(\Omega\cup\{u<t\})}u^2\notag\\
&\le \int_{D}|\nabla (u\vee t)|^2\,dx+\beta t^2\, \text{\rm Per}(\{u<t\})+\beta\int_{\{u> t\}\cap \partial^\ast\Omega}u^2,\notag
\end{align}
which proves that $u$ satisfies the optimality condition \eqref{e:opt_cond} from Lemma \ref{l:opt_cond}: 
\begin{equation}\label{e:opt_cond_intro}
\int_{\{u<t\}}|\nabla u|^2\,dx\le \beta\, t^2\, \text{\rm Per}\big(\{u<t\}\big).
\end{equation}
Now, applying Proposition \ref{p:non-degeneracy}, we get that $u$ is bounded from below by a strictly positive constant in $D$. Finally, Proposition \ref{p:holder} gives that $u$ is H\"older continuous in $D$. This proves Theorem \ref{t:main} (ii). 

\subsubsection{Regularity of the free boundary} In order to prove the regularity of the free boundary (Theorem \ref{t:main} (iii)), we use the H\"older continuity and the non-degeneracy of $u$ to show that a solution $\Omega$ is an almost-minimizer of the perimeter. We do this in Theorem \ref{t:regularity}.  Now, from the classical regularity theory for almost-minizers of the perimeter (see \cite{tamanini}), we obtain that (inside $D$) the free boundary $\partial\Omega$ can be decomposed into a $C^{1,\alpha}$-regular part $\text{Reg}(\partial\Omega)$ and a (possibly empty) singular part of Hausdorff dimension smaller than $d-8$.

Finally, in Theorem \ref{t:higher_regularity}, we prove the $C^\infty$ regularity of $\text{Reg}(\partial\Omega)$. In order to do so, we first show (see Lemma \ref{l:robin_condition}) that in a neighborhood of a regular point $x_0$, the restrictions $u_+$ and $u_-$ of $u$ on $\Omega$ and $D\setminus\Omega$ are solutions of the following transmission problem: 
$$\begin{cases}
\Delta u_+=0\quad\text{in}\quad \Omega,\\
\Delta u_-=0\quad\text{in}\quad D\setminus\overline\Omega,\\
u_+=u_-=u\quad\text{on}\quad\partial\Omega,\\
\ds\frac{\partial u_+}{\partial \nu_\Omega}-\frac{\partial u_-}{\partial \nu_\Omega}+2\beta u=0\quad\text{on}\quad \partial\Omega,
\end{cases}$$
where $\nu_\Omega$ is the normal derivative to $\partial\Omega$. Now, using the recent results \cite{css} and \cite{D}, we get that $u_+$ and $u_-$ are as regular as the free boundary $\partial\Omega$ (see Lemma \ref{l:reg1}). On the other hand, using variations of $u$ along smooth vector fields, we obtain that $\text{Reg}(\partial\Omega)$ solves an equation of the form 
$$\text{``Mean curvature of $\partial\Omega$"}=F(\nabla u_+,\nabla u_-,u_\pm)\quad\text{on}\quad \partial\Omega,$$
where $F$ is an explicit (rational) function of $\nabla u_\pm$ and $u$.
In particular, this implies that $\partial\Omega$ gains one more derivative with respect to $u$, that is, $u\in C^{k,\alpha}\,\Rightarrow\,\partial\Omega\in C^{k+1,\alpha}$. Thus, by a bootstrap argument, we get that the regular part of the free boundary is $C^\infty$.

\subsection{On the non-degeneracy of the solutions}
We notice that the competitors $(u_t,\Omega_t)$ in \eqref{e:fake_opt_cond} are the two-phase analogue of the ones used by Caffarelli and Kriventsov in \cite{ck}, where the authors study a one-phase version of \eqref{e:pb}. Nevertheless, the functional in  \cite{ck} involves the measure of $\Omega$, which means that the optimality condition there corresponds to
\begin{equation*}
J_\beta(u,\Omega)+\overline C|\Omega\cap\{u\le t\}|\le J_\beta(u_t,\Omega_t)\quad\text{where}\quad u_t=u\vee t\quad\text{and}\quad \Omega_t=\Omega\setminus  \{u\le t\},
\end{equation*}
where $\bar C>0$. The presence of the constant $\bar C$ allows to prove the bound from below by using a differential inequality for a suitably chosen function $f(t)$, which is given in terms of $u$ and $\{u<t\}$ (see Proposition \ref{p:non-degeneracy} and \cite[Theorem 3.2]{ck}). In Proposition \ref{p:non-degeneracy}, we exploit the same idea, but since we do not have the constant $\bar C$, we can only conclude that $f(t)\ge \eps t$ (which is not in contradiction with the fact that $f(t)$ is defined for every $t>0$). So, we continue, and we use this lower bound to obtain a bound of the form 
\begin{equation}\label{e:non-deg-intro-0}
c\le \beta^{\sfrac12}\text{Per}(\{u<t\})^{\sfrac12}|\{u<t\}|^{\sfrac12}\quad\text{for every}\quad t>0,
\end{equation}
where $u:=u_\infty$ and $c$ is a constant depending on $\beta$ and $d$. Then, we notice that this entails 
$$c\le \beta^{\sfrac34}\text{Per}(\{u<t\})^{\sfrac14}|\{u<t\}|^{\sfrac34}\quad\text{for every}\quad t>0.$$
and we use an iteration procedure to get that 
$$c\le \beta^{1-\sfrac1{2^n}}\text{Per}(\{u<t\})^{\sfrac1{2^n}}|\{u<t\}|^{1-\sfrac1{2^n}}\quad\text{for every}\quad t>0.$$
Passing to the limit as $n\to\infty$, we get that if $u$ is not bounded away from zero, then
\begin{equation}\label{e:non-deg-intro}
c\le \beta |\{u<t\}|\le\beta |D|\quad\text{for every}\quad t>0.
\end{equation}
Now, this means that the measure of the zero-set $|\{u=0\}|$ is bounded from below. Thus, using again the optimality of $u$, we get that \eqref{e:non-deg-intro-0} holds with an arbitrary small $\eps>0$ in place of $\beta$, we get that `
\begin{equation*}
c\le \eps |\{u<t\}|\quad\text{for every}\quad t>0,
\end{equation*}
which is impossible.

A similar non-degeneracy result was proved by Bucur and Giacomini in \cite{bg} by a De Giorgi iteration scheme\footnote{We are grateful to the anonymous referee for bringing  to our attention the reference \cite{bg}.}. Precisely, one can prove that any solution to \eqref{e:pb} satisfies the optimality condition from \cite[Remark 3.7]{bg}. Thus, \cite[Theorem 3.5]{bg} also to the solutions of \eqref{e:pb}. Conversely, the argument from \ref{p:non-degeneracy} can be applied to the minimizers from \cite{bg} to obtain the bound from below from \cite[Theorem 3.5]{bg}.

\subsection{One-phase and two-phase problems with Robin boundary conditions}
The problem \eqref{e:pb} is the first instance of a two-phase free boundary problem with Robin boundary conditions. Precisely, we notice that if $\Omega$ is a fixed set with smooth boundary and if $u$ minimizes the functional $J_\beta(\cdot,\Omega)$ in $H^1(D)$, then the functions 
$$u_+:=u\quad\text{on}\quad \overline\Omega\qquad\text{and}\qquad u_-:=u\quad\text{on}\quad D\setminus\Omega,$$
are harmonic respectively in $\Omega$ and $D\setminus\overline\Omega$, and satisfy the following conditions on $\partial\Omega\cap D$:
\begin{equation}\label{e:two-phase-optimality}
u_+=u_-\quad\text{and}\quad\left(\frac{\partial u_+}{\partial \nu_+}+\frac\beta2 u_+\right)+\left(\frac{\partial u_-}{\partial \nu_-}+\frac\beta2u_-\right)=0\quad\text{on}\quad \partial\Omega\cap D,
\end{equation}
where $\nu_+$ and $\nu_-$ are the exterior and the interior normals to $\partial\Omega$.  Notice that \eqref{e:two-phase-optimality} is a two-phase counterpart of the one-phase problem 
\begin{equation}\label{e:one-phase-optimality}
\Delta u=0\quad\text{in}\quad\Omega,\qquad \frac{\partial u}{\partial \nu}+\beta u=0\quad\text{on}\quad \partial\Omega\cap D,
\end{equation}
which was studied by Bucur-Luckhaus in \cite{bl} and Caffarelli-Kriventsov in \cite{ck}. As explained in \cite{ck}, the Robin condition in \eqref{e:one-phase-optimality} naturally arises in the physical situation in which the heat diffuses freely in $\Omega$, the temperature is set to be zero on the surface $\partial\Omega$, which is separated from the interior of $\Omega$ by an inifinitesimal insulator. The two-phase problem \eqref{e:two-phase-optimality} also may be interpreted in this way, in this case the heat diffuses freely both inside $\Omega$ and outside, in $D\setminus\overline\Omega$; the temperature is set to be zero on the surface $\partial\Omega$, which is insulated from both sides; the continuity of the temperature means that the heat transfer is allowed also across $\partial\Omega$, which happens for instance if the surface $\partial\Omega$ is replaced by a very thin (infinitesimal) net. 
\medskip

Even if the problems in \cite{bl}-\cite{ck} and in the present paper lead to the free boundary conditions of the same type, the techniques are completely different. For instance, the problem studied in  \cite{bl}-\cite{ck} is a free discontinuity problem as the function $u$ jumps from positive in $\Omega$ to zero in $D\setminus \overline\Omega$. Thus, the corresponding variational minimization problem can be naturally stated in the class of SBV functions, which clearly influences both the existence and the regularity techniques; roughly speaking, the existence is obtained through a compactness theorem in the SBV class, while the regularity relies on techniques related to the Mumford-Shah functional. 

In our case, the problem can be stated for the functions $(u_1,u_2)$ with disjoint supports ($u_1u_2=0$ almost-everywhere in $D$) which satisfy the following constraints: the sum $u_1+u_2$ should be a Sobolev function (this corresponds to the continuity condition in \eqref{e:two-phase-optimality}); $u_1^2$ and $u_2^2$ are SBV funtions whose jump sets are contained in the boundary of the positivity sets $\{u_1>0\}$ and $\{u_2>0\}$. Now, it is reasonable to expect that an existence result can be proved also in this class, but then, in order to prove that a solution to \eqref{e:pb} exists, one should show that $u_1$ and $u_2$ are of the form $u_1=u\ind_{\Omega}$ and $u_2=u\ind_{D\setminus \Omega}$ for a set of finite perimeter $\Omega\subset\R^d$, $u$ being the sum $u_1+u_2$. Summarizing, working in the class of SBV functions would allow to state \eqref{e:pb} in a weaker form, but it doesn't seem to be a shortcut to the existence of a solution (of \eqref{e:pb}) as it will require the analysis of the jump sets of the optimal couples in the SBV class. 
 Thus, we prefer not to rely on the advanced compactness results for SBV functions, but to prove the existence of a solution from scratch. 
 
  Finally, as explained in Section \ref{sub:sketch}, once we know that an optimal couple $(u,\Omega)$ exists, and that $u$ is  non-degenerate and H\"older continuous, the regularity of the free boundary $\partial\Omega$ follows immediately since the set $\Omega$ becomes an almost-minimizer of the perimeter.


\section{Preliminaries}\label{s:prelimiaries}
\subsection{Sets of finite perimeter}\label{sub:perimeter}
Let $A\subset\R^d$ be a an open set in $\R^d$. We recall that the set $E\subset\R^d$ is said to have a {\it finite perimeter in $A$} if 
\begin{equation}\label{e:per}\text{Per}(E,A)=\sup\Big\{\int_{A}\text{div}\,\xi(x)\,dx\ :\ \xi\in C^1_c(A;\R^d),\ \sup_{x\in \R^d}|\xi (x)|\le 1\Big\},
\end{equation}
is finite. We say that $E$ has a {\it locally finite perimeter in $A$}, if for every open set $B\subset \R^d$ such that $\overline B\subset A$, we have that $\text{Per}(E,B)<\infty$. We say that {\it $E$ is of finite perimeter} if 
$$\text{Per}(E):=\text{Per}(E,\R^d)<+\infty.$$
By the De Giorgi structure theorem (see for instance \cite[Theorem II.4.9]{maggi}), if the set $E\subset \R^d$ has locally finite perimeter in $A$, then there is a set $\partial^\ast E\subset A\cap \partial E$ called {\it reduced boundary} such that 
$$\text{Per}(E,B)=\HH^{d-1}(B\cap\partial^\ast E)\quad\text{for every set}\quad B\subset\subset A,$$
where $\HH^{d-1}$ is the $(d-1)$-dimensional Hausdorff measure in $\R^d$.
Moreover, there is a $\HH^{d-1}$-measurable map $\nu_E:\partial^\ast E\to\R^d$, called {\it generalized normal} such that $|\nu_E|=1$ and 
$$\int_E\text{div}\,\xi(x)\,dx=\int_{\partial^\ast E}\nu_E\cdot \xi\,d\HH^{d-1}\quad\text{for every}\quad \xi\in C^1_c(A;\R^d).$$
\subsection{Capacity and traces of Sobolev functions}
We define the capacity (or the $2$-capacity) of a set $E\subset\R^d$ as
$$\cp (E)=\inf\Big\{\|u\|_{H^1(\R^d)}^2\ :\ u\in H^1(\R^d),\ u\ge 1\ \text{in a neighborhood of}\ E\Big\}.$$
Suppose now that $d\ge 3$. It is well-known that the sets of zero capacity have zero $d-1$ dimensional Hausdorff measure (see for instance \cite[Section 4.7.2, Theorem 4]{eg}), that is, 
$$\text{If}\quad \cp(E)=0\ ,\quad\text{then}\quad \HH^{d-1}(E)=0.$$
Moreover, the Sobolev functions are defined up to a set of zero capacity (i.e. {\it quasi-everywhere}), that is, if $A\subset \R^d$ is an open set and $u\in H^1(A)$, then there is a set $\mathcal N_u\subset\R^d$ such that $\cp\,(\mathcal N_u)=0$ and 
$$u(x_0)=\lim_{r\to0}\frac1{|B_r|}\int_{B_r(x_0)}u(x)\,dx\qquad\text{for every}\qquad x_0\in A\setminus \mathcal N_u.$$
Moreover, for every function $u\in H^1(A)$ there is a sequence $u_n\in C^\infty(A)\cap H^1(A)$ and a set $\mathcal N\subset A$ of zero capacity such that: 
\begin{itemize}
\item $u_n$ converges to $u$ strongly in $H^1(A)$;
\item $\ds u(x)=\lim_{n\to\infty}u_n(x)$ for every $x\in A\setminus(\mathcal N\cup\mathcal N_u)$.
\end{itemize}
In particular, if $E\subset\R^d$ is a set of locally finite perimeter in the open set $A\subset\R^d$ and if $u\in H^1(A)$, then the function $u^2$ is defined $\HH^{d-1}$-almost everywhere on $\partial^\ast E$ and is $\HH^{d-1}$ measurable on $\partial^\ast E$. Thus, the integral 
$$\mathcal I(u,E):=\int_{A\cap \partial^\ast E}u^2\,d\HH^{d-1}\quad\text{is well-defined}.$$
As a consequence of the discussion above, we have the following proposition. 
\begin{prop}\label{prop0}
Suppose that $D$ is a smooth bounded open set in $\R^d$ and that $u\in H^1(D)$ is a  Sobolev function. Then, there is a set $\mathcal N_u\subset D$ such that $\HH^{d-1}(\mathcal N_u)=0$ and  $$u(x_0)=\lim_{r\to0}\frac1{|B_r|}\int_{B_r(x_0)}u(x)\,dx\qquad\text{for every}\qquad x_0\in D\setminus \mathcal N_u.$$ 
Moreover, if $E\subset\R^d$ is a set of locally finite perimeter in $\R^d$, then the function $u:\partial^\ast E\cap D\to\R$ is defined $\HH^{d-1}$-almost everywhere and is  $\HH^{d-1}$-measurable on $\partial^\ast E\cap D$. In particular, the integral
$\ds \mathcal I(u,D)$ is well-defined.
\end{prop}	
\begin{oss}
In the case $d=2$, \eqref{prop0} still holds. In fact, it is sufficient to notice that if $u\in H^1(D)$, then $u\in W^{1,p}(D)$ for any $1<p<2$. In particular, it is sufficient to results from \cite{eg}, this time in the space $W^{1,p}(D)$, for $p$ close to $2$.  
\end{oss}
In the next subsection, we will go through the main properties of this functional, which we will need in the proof of Theorem \ref{t:main}.
\subsection{Properties of the functional $\mathcal I$}
We first notice that we can use an integration by parts to write $\mathcal I$ as in \eqref{e:per}.
\begin{lemma}
Let $E\subset\R^d$ be a set of locally finite perimeter in the open set $A\subset\R^d$ and let $u\in H^1(A)$ be locally bounded in $A$. Then, the following holds.
\begin{enumerate}[\rm (i)]
\item For every $\xi\in C^1_c(A;\R^d)$ we have 	 
\begin{equation}\label{e:ipp}
\int_{A\cap\partial^\ast E}(\xi\cdot\nu_E)\, u^2\,d\HH^{d-1}=\int_{A}\text{\rm div}\,(u^2\xi)\,dx\ .
\end{equation}
\item We have the formula
\begin{equation}\label{e:formula}
\int_{A\cap\partial^\ast E}u^2\,d\HH^{d-1}=\sup\left\{\int_{A\cap\partial^\ast E}(\xi\cdot\nu_E)\, u^2\,d\HH^{d-1}\ :\ \xi\in C^1_c(A;\R^d),\ |\xi|\le 1\right\}.
\end{equation}
\end{enumerate}
\end{lemma}
\begin{proof}
The first claim follows by a classical approximation argument with functions of the form $\phi_n\ast u$, where $\phi_n$ is a sequence of mollifiers. In order to prove claim (ii), we notice that 
\begin{equation*}
\int_{A\cap\partial^\ast E}u^2\,d\HH^{d-1}\le \sup\left\{\int_{A\cap\partial^\ast E}(\xi\cdot\nu_E)\, u^2\,d\HH^{d-1}\ :\ \xi\in C^1_c(A;\R^d),\ |\xi|\le 1\right\}.
\end{equation*}
Thus, it is sufficient to find a sequence $\xi_n\in C^1_c(A;\R^d)$, $|\xi_n|\le 1$, such that 
$$\int_{A\cap\partial^\ast E}u^2\,d\HH^{d-1}= \lim_{n\to\infty}\int_{A\cap\partial^\ast E}(\xi_n\cdot\nu_E)\, u^2\,d\HH^{d-1}.$$
Let $A_n$ be a sequence of open sets such that $A_n\subset\subset A$ and $\ind_{A_n}\to\ind_A$. 
Then 
$$\int_{A\cap\partial^\ast E}u^2\,d\HH^{d-1}= \lim_{n\to\infty}\int_{A_n\cap\partial^\ast E}u^2\,d\HH^{d-1}.$$
Setting $\ds M_n=\sup_{A_n} u^2$, we can find $\xi_n\in C^1_c(A;\R^d)$ such that $|\xi_n|\le 1$, and 
$$0\le \text{Per}(E,A_n)-\int_{\partial^\ast E}(\xi_n\cdot\nu_E)\,d\HH^{d-1}\le \frac1{nM_n}.$$
In particular, this implies that
$$0\le \int_{A_n\cap\partial^\ast E}u^2\,d\HH^{d-1}-\int_{A_n\cap\partial^\ast E}(\xi_n\cdot \nu_E)u^2\,d\HH^{d-1}\le \frac1n,$$
which concludes the proof.
\end{proof}	
\begin{lemma}[Main semicontinuity lemma]\label{l:semicontinuity}
	Suppose that $A\subset\R^d$ is a bounded open set and that $h:A\to \R$ is a non-negative function in $L^1(A)$. Let $u_n\in H^1(A)$ be a sequence of functions and $\Omega_n\subset \R^d$ be a sequence of sets of locally finite perimeter in $A$ such that: 
	\begin{enumerate}[\quad (a)]
		\item $0\le u_n\le h$ in $A$, for every $n\in\N$;	
		\item there is a function $u_\infty\in H^1(A)$ such that $u_n$ converges to $u_\infty$ weakly in $H^1(A)$ and pointwise almost-everywhere in $A$;
		\item there is a set $\Omega_\infty\subset \R^d$ of locally finite finite perimeter in $A$ such that the sequence of characteristic functions $\ind_{\Omega_n}$ converges to $\ind_{\Omega_\infty}$ pointwise almost-everywhere in $A$.
	\end{enumerate}	
	Then, 
	\begin{equation}\label{e:semicontinuity}
	\int_{A\cap\partial^\ast\Omega_\infty}u_\infty^2\,d\HH^{d-1}\le \liminf_{n\to\infty} \int_{A\cap\partial^\ast\Omega_n}u_n^2\,d\HH^{d-1}.
	\end{equation}
\end{lemma}	
\begin{proof}
	Notice that, for every $u\in H^1(A)$ and every set of finite perimeter $\Omega$, we have 
	$$\int_{A\cap\partial^\ast \Omega}u^2\,d\HH^{d-1}=\sup\left\{\int_{A\cap\partial^\ast \Omega}(\xi\cdot\nu_\Omega)\, u^2\,d\HH^{d-1}\ :\ \xi\in C^1_c(A;\R^d),\ |\xi|\le 1\right\},$$
	where $\nu_\Omega$ denotes the exterior normal to $\partial^\ast\Omega$. We use the notation 
	$$\nu_n:=\nu_{\Omega_n}\qquad\text{and}\qquad \nu_\infty:=\nu_{\Omega_\infty}.$$
	Let now $ \xi\in C^1_c(A;\R^d),\ |\xi|\le 1$ be fixed. By the divergence theorem, we have
	\begin{align*}
	\liminf_{n\to\infty}\int_{A\cap\partial^\ast \Omega_n}u_n^2\,d\HH^{d-1}&\ge	
	\liminf_{n\to\infty}\int_{A\cap\Omega_n}\text{div}\big(u_n^2\xi\big)\,dx = 	\liminf_{n\to\infty}\int_{A}\ind_{\Omega_n}\text{div}\big(u_n^2\xi\big)\,dx \\
	&= 	\liminf_{n\to\infty}\int_{A}\Big(2\big(u_n\ind_{\Omega_n}\xi\big)\cdot \nabla u_n+\big(u_n\,\ind_{\Omega_n}\big)\,(u_n\,\text{div}\,\xi)\Big)\,dx\,\\
	&= 	\int_{A}\Big(2\big(u_\infty\ind_{\Omega_\infty}\xi\big)\cdot \nabla u_\infty+\big(u_\infty\,\ind_{\Omega_\infty}\big)\,(u_\infty\,\text{div}\,\xi)\Big)\,dx\,\\
	&=\int_{A\cap\Omega_\infty}\text{div}\big(u_\infty^2\xi\big)\,dx=\int_{A\cap \partial^\ast \Omega_\infty}(\xi\cdot\nu_\infty)\,u_\infty^2\,d\HH^{d-1},
	\end{align*}
	where in order to pass to the limit we used that the sequence $u_n\ind_{\Omega_n}$ converges strongly in $L^2_{loc}(A)$ to $u_\infty\ind_{\Omega_\infty}$, as a consequence of the fact that it converges pointwise a.e. and is bounded by $h$. Now, taking the supremum over $\xi$, we get \eqref{e:semicontinuity}. 
\end{proof}	

\section{A family of approximating problems}\label{s:approx}
We use the notations $D,\beta,E,v, \mathcal E,\mathcal V$ from Section \ref{s:intro}. Moreover, we fix a constant 
$$\eps\in[0,m),$$ 
where $m$ is the lower bound of the function $v$, and we consider the auxiliary problem 
\begin{equation}\label{e:pb_aux}
\min\big\{J_\beta(\Omega,u)\ :\ \Omega\in \mathcal E,\ u\in \mathcal V,\ u\ge \eps\ \text{in}\ \R^d\big\}.
\end{equation}

\begin{prop}[Existence of a solution]\label{p:existence:aux}
Let $\mathcal E$ and $\mathcal V$ be as above.
Then, for every $0<\eps<m$, there is a solution to the problem \eqref{e:pb_aux}.  
\end{prop}	
\begin{proof}
Let $(u_n,\Omega_n)$ be a minimizing sequence for \eqref{e:pb_aux}. Since 
$$\int_{D}|\nabla u_n|^2\,dx+\int_{\partial^\ast \Omega_n}u_n^2\,d\HH^{d-1}= J_\beta(u_n,\Omega_n)\le J_\beta(v,E),$$
for every $n\in\N$, we have 
$$\int_{D}|\nabla u_n|^2\,dx\le J_\beta(v,E)\qquad\text{and}\qquad \text{\rm Per}(\Omega_n)\le \frac{1}{\beta \eps^2}  J_\beta(v,E).$$
Thus, there are subsequences $u_n$ and $\Omega_n$ such that:
\begin{itemize}
\item $u_n$ converges strongly in $L^2(D)$, weakly in $H^1(D)$ and pointwise almost-everywhere to a function $u_\infty\in H^1(D)$;
\item $\ind_{\Omega_n}$ converges to $\ind_{\Omega_\infty}$ strongly in $L^1(D)$ and pointwise almost-everywhere.
\end{itemize}
Moreover, we can assume that $u_n\le h$ on $D$, where $h$ is the harmonic function :
$$\Delta h=0\quad\text{in}\quad D\ ,\qquad h-v\in H^1_0(D).$$
Indeed, we have 
\begin{align*}
\int_D|\nabla u_n|^2\,dx=\int_D|\nabla (u_n\wedge h)|^2\,dx&+\int_D|\nabla (u_n\vee h)|^2\,dx-\int_D|\nabla h|^2\,dx\ge \int_D|\nabla (u_n\wedge h)|^2\,dx\ ,\\
\text{and}\qquad \int_{\partial^\ast\Omega}u_n^2\,d\HH^{d-1}&\ge \int_{\partial^\ast\Omega}(u_n\wedge h)^2\,d\HH^{d-1}\ ,
\end{align*}
which gives that 
$$J_\beta (u_n\wedge h,\Omega_n)\le J_\beta (u_n,\Omega_n).$$
On the other hand, we have that 
$$J_\beta(u_n\vee0,\Omega_n)\le J_\beta(u_n,\Omega_n).$$
Thus, we can assume that $0\le u_n\le h$, for every $n\in\N$, and so the hypotheses of Lemma \ref{l:semicontinuity} are satisfied, which means that \eqref{e:semicontinuity} holds.
Moreover, by the semicontinuity of the $H^1$ norm we have 
$$\int_D|\nabla u_\infty|^2\,dx\le \liminf_{n\to\infty}\int_D|\nabla u_n|^2\,dx,$$
which finally implies that 
$$\qquad\qquad\qquad\qquad\qquad\qquad  J_\beta (u_\infty,\Omega_\infty)\le\liminf_{n\to\infty} J_\beta (u_n,\Omega_n).\qquad\qquad\qquad\qquad\qquad\quad \qedhere$$
\end{proof}

\begin{lemma}[Subharmonicity of the solutions]\label{l:subharmonicity}
Let $m>0$, $\beta>0$ and $\eps\in[0,m)$ be fixed. Let the function $u_\eps\in H^1(D)$ and the set of finite perimeter $\Omega_\eps$ be such that the couple $(u_\eps,\Omega_\eps)$ is a solution to the problem \eqref{e:pb_aux}. Then $u_\eps$ is subharmonic in $D$ and there is a positive Radon measure $\mu_\eps$ such that 
$$-\int_{D}\nabla u_\eps\cdot\nabla \varphi\,dx=\int_{D}\varphi\,d\mu_\eps\quad\text{for every}\quad \varphi\in H^1_0(D).$$
\end{lemma}	
\begin{oss}
$\mu_\eps$ is the distributional Laplacian of $u_\eps$. We will use the notation $\mu_\eps=\Delta u_\eps$.
\end{oss}
\begin{proof}
Let $\varphi\le u_\eps$ be a function in $H^1(D)$ such that $\varphi=u_\eps$ on $\partial D$. Then, testing the optimality of $(u_\eps,\Omega_\eps)$ with $(\varphi\vee\eps,\Omega_\eps)$ and using the fact that $u_\eps\ge \varphi\vee\eps$, we get 
\begin{align*}
\int_D|\nabla \varphi|^2\,dx&\ge \int_D|\nabla (\varphi\vee \eps)|^2\,dx\\
&\ge \int_D|\nabla u_\eps|^2\,dx +\int_{\partial^\ast\Omega_\eps}u_\eps^2\,d\HH^{d-1}-\int_{\partial^\ast\Omega_\eps}(\varphi\vee \eps)^2\,d\HH^{d-1}\ge  \int_D|\nabla u_\eps|^2\,dx,
\end{align*} 	
which concludes the proof.
\end{proof}

We will next show that the family of solutions $\big\{u_\eps\big\}_{\eps\in(0,m)}$ is uniformly H\"older continuous. 
We will use the following lemma, which can be proved in several different ways. Here, we give a short proof based on the mean-value formula for subharmonic functions. Similar argument was used to prove the Lipschitz continuity of the solutions to some free boundary problems (see for instance \cite{velectures} and the references therein).

\begin{lemma}[A general condition for the H\"older continuity]\label{l:holder}
Let $D$ be a bounded open set in $\R^d$ and let $h\in L^\infty_{loc}(D)$. 
Suppose that $u\in H^1(D)$ is such that 
\begin{enumerate}[\quad(a)]
\item $0\le u\le h$ in $D$;
\item $u$ is subharmonic in $D$; 
\item there are constants $K>0$ and $\alpha\in[0,1)$ such that 
\begin{equation}\label{e:opt_cond_balls}
\Delta u\big(B_r(x_0)\big)\le K\,r^{d-1-\alpha}\quad\text{for every}\quad x_0\in \overline D_\delta\quad\text{and every}\quad 0<r<\frac\delta2,
\end{equation}
where $\delta>0$ and 
\begin{equation}\label{e:Ddelta}
D_\delta:=\big\{x\in D\ :\ \text{\rm dist}(x,\partial D)> \delta\big\}.
\end{equation}
\end{enumerate}	 
Then, there is a constant $C$ depending on $\delta$, $h$, $\alpha$ and $K$ such that 
$$|u(x)-u(y)|\le C|x-y|^{\frac{1-\alpha}{2-\alpha}}\qquad\text{for every}\qquad x,y\in \overline D_\delta.$$
\end{lemma}	
\begin{proof}
We first notice that the following formula is true for every subharmonic function $u\in H^1(D)$ and for every $x_0\in D$ and $0<s<t<\text{dist}(x_0,\partial D)$.
$$\mean{\partial B_t(x_0)}{u\,d\HH^{d-1}}-\mean{\partial B_s(x_0)}{u\,d\HH^{d-1}}=\frac{1}{d\omega_d}\int_s^t r^{1-d}\Delta u\big(B_r(x_0)\big)\,dr.$$
In particular, the function 
$$r\mapsto \mean{\partial B_r(x_0)}{u\,d\HH^{d-1}},$$
is monotone and we can define the function $u$ pointwise {\it everywhere} as
$$u(x_0):=\lim_{r\to0}\mean{\partial B_r(x_0)}{u\,d\HH^{d-1}}.$$
As a consequence, for every $R<\text{dist}(x_0,\partial D)$, we have
$$\mean{\partial B_R(x_0)}{u\,d\HH^{d-1}}-u(x_0)=\frac{1}{d\omega_d}\int_0^R r^{1-d}\Delta u\big(B_r(x_0)\big)\,dr.$$
Now, applying \eqref{e:opt_cond_balls}, and integrating in $r$, we get that if $x_0\in \overline D_\delta$ and $R<\frac\delta2$, then
\begin{equation}\label{e:mean_continuity1}
0\le \mean{\partial B_R(x_0)}{u\,d\HH^{d-1}}-u(x_0)\le C\,R^{1-\alpha}\qquad\text{where}\qquad C:=\frac{K}{d\omega_d(1-\alpha)},
\end{equation}
which, by the subharmonicity of $u$, implies
\begin{equation}\label{e:mean_continuity2}
0\le \mean{B_R(x_0)}{u\,dx}-u(x_0)\le C\,R^{1-\alpha}.
\end{equation}
Let now $x_0,y_0\in \overline D_\delta$ be such that 
$$|x_0-y_0|\le 1\qquad\text{and}\qquad R:=|x_0-y_0|^{\gamma}\le \frac\delta4,$$
where $\gamma\in(0,1)$ will be chosen later.

Now, since $B_{R}(x_0)\subset B_{R+|x_0-y_0|}(y_0)\subset B_{2R}(y_0)\subset D$, we can estimate 
\begin{align*}
u(x_0)-u(y_0)&\le  \mean{B_R(x_0)}{u(x)\,dx}-u(y_0)\le  \frac{\big(R+|x_0-y_0|\big)^d}{R^d}\mean{B_{R+|x_0-y_0|}(y_0)}{u(x)\,dx}-u(y_0)\\
&=  \left(1+|x_0-y_0|^{1-\gamma}\right)^d\mean{B_{R+|x_0-y_0|}(y_0)}{u(x)\,dx}-u(y_0)\\
&\le  \left(1+d2^{d-1}|x_0-y_0|^{1-\gamma}\right)\mean{B_{R+|x_0-y_0|}(y_0)}{u(x)\,dx}-u(y_0),
\end{align*}
where in the last inequality we used that $|x_0-y_0|^{1-\gamma}\le 1$. Now, using \eqref{e:mean_continuity2}, we get 
\begin{align*}
u(x_0)-u(y_0)
&\le  \Big(\mean{B_{R+|x_0-y_0|}(y_0)}{u(x)\,dx}-u(y_0)\Big)+d2^{d-1}|x_0-y_0|^{1-\gamma}\|u\|_{L^\infty(B_{2R}(y_0))}\\
&\le  C\big(R+|x_0-y_0|\big)^{1-\alpha}+d2^{d-1}|x_0-y_0|^{1-\gamma}\|h\|_{L^\infty(B_{2R}(y_0))}\\
&\le  \left(2C+d2^{d-1}M_{\sfrac\delta2}\right)|x_0-y_0|^{1-\gamma},
\end{align*}
where $M_{\sfrac\delta2}$ is the maximum of $h$ on the set $\overline D_{\sfrac\delta2}$ and where we choose $\ds\gamma=\frac1{2-\alpha}$, which implies that 
$\gamma(1-\alpha)=1-\gamma$ and $\ds1-\gamma=\frac{1-\alpha}{2-\alpha}$.
\end{proof}

\begin{prop}[H\"older continuity of the solution]\label{p:holder}
Let $m>0$, $\beta>0$ and $\eps\in[0,m)$ be fixed. Let the function $u_\eps\in H^1(D)$ and the set of finite perimeter $\Omega_\eps$ be such that the couple $(u_\eps,\Omega_\eps)$ is a solution to the problem \eqref{e:pb_aux} with some $v\in H^1(D)$ and $E\subset \R^d$. Then, for every $\delta>0$, there is a constant $C$ depending on $D$, $\delta$ and $v$ (but not on $\eps$) such that 
$$|u_\eps(x)-u_\eps(y)|\le C|x-y|^{\sfrac13}\qquad\text{for every}\qquad x,y\in D_\delta.$$
\end{prop}	
\begin{proof}
By Lemma \ref{l:subharmonicity}, we have that $u_\eps$ is subharmonic and, in particular, $0\le u_\eps\le h$ in $D$, where $h$ is the harmonic extension of $v$ in $D$. Thus, it is sufficient to prove that \eqref{e:opt_cond_balls} holds. Let $x_0\in \overline D_\delta$ and $R\le \frac\delta2$. Let $\varphi\in C^\infty_c(B_{\sfrac{3R}2}(x_0))$ be such that 
$$\varphi=1\quad\text{on}\quad B_R(x_0),\qquad |\nabla \varphi|\le \frac{3}{R}\quad\text{in}\quad B_{\sfrac{3R}2}(x_0).$$
Now, we test the optimality of $(u_\eps,\Omega_\eps)$ with $(\widetilde u_\eps,\widetilde \Omega_\eps)$, where 
$$\widetilde u_\eps=u_\eps+R^{\sfrac12}\varphi\qquad\text{and}\qquad \widetilde \Omega_\eps=\Omega_\eps\cup B_{\sfrac{3R}2}(x_0).$$
Thus, we get 
\begin{align*}
\int_{D}|\nabla u_\eps|^2\,dx+\beta\int_{\partial^\ast\Omega_\eps}u_\eps^2\,d\HH^{d-1}&\le \int_{D}|\nabla \widetilde u_\eps|^2\,dx+\beta\int_{\partial^\ast\widetilde \Omega_\eps}\widetilde u_\eps^2\,d\HH^{d-1}
\end{align*}
and
\begin{align*}
\int_{\partial^\ast\widetilde \Omega_\eps}\widetilde u_\eps^2\,d\HH^{d-1}&\le \int_{\partial^\ast \Omega_\eps\setminus B_{\sfrac{3R}2}(x_0)} u_\eps^2\,d\HH^{d-1}+\int_{\partial B_{\sfrac{3R}2}(x_0)} u_\eps^2\,d\HH^{d-1}\\
&\le \int_{\partial^\ast\Omega_\eps}u_\eps^2\,d\HH^{d-1}+C_d R^{\,d-1}M_{\sfrac{\delta}4}^2,
\end{align*}
where $M_\rho:=\sup\big\{h(x)\,:\,x\in \overline D_\rho\big\}.$
Thus, we obtain 
\begin{align*}
2R^{\sfrac12}\int_{B_{\sfrac{3R}2}(x_0)}\!\!\!-\nabla u_\eps\cdot\nabla \varphi\,dx&\le R\int_{B_{\sfrac{3R}2}(x_0)}|\nabla \varphi|^2\,dx+\beta C_dR^{\,d-1}M_{\sfrac{\delta}4}^2\le C_d\big(1+\beta M_{\sfrac{\delta}4}^2\big)R^{\,d-1}\ ,
\end{align*}
which implies that 
$$\Delta u_\eps\big(B_R(x_0)\big)\le C_d\big(1+\beta M_{\sfrac{\delta}4}^2\big)R^{\,d-\sfrac32},$$
which concludes the proof of \eqref{e:opt_cond_balls} with $\alpha=\sfrac12$\,.
\end{proof}	

\section{Existence of an optimal set}\label{s:existence}

\subsection{Definition of $(u_0,\Omega_0)$}\label{sub:construction}
Now, for any $\eps\in(0,m)$, we consider the solution $(u_\eps,\Omega_\eps)$ of \eqref{e:pb_aux}. As a consequence of Proposition \eqref{p:holder}, we can find a sequence $\eps_n\to0$ and a function $u_0\in H^1(D)\cap C^{0,\sfrac13}(D)$ such that :
\begin{itemize}
\item $u_{\eps_n}$ converges to $u_0$ uniformly on every set $D_\delta$, $\delta>0$, where $D_\delta$ is defined in \eqref{e:Ddelta};
\item $u_{\eps_n}$ converges to $u_0$ strongly in $L^2(D)$; 
\item $u_{\eps_n}$ converges to $u_0$ weakly in $H^1(D)$.
\end{itemize}	
Our aim in this section is to show that $u_0$ is a solution to \eqref{e:pb}. 


The construction of $\Omega_0$ is more delicate. First, we fix $t>0$ and $\delta>0$ and we notice that the perimeter of $\Omega_{\eps_n}$ is bounded on the open set $\{u_0>t\}\cap D_\delta$. Indeed, the uniform convergence of $u_{\eps_n}$ to $u_0$ implies that, for $n$ large enough ($n\ge N_{t,\delta}$, for some fixed $N_{t,\delta}\in\N$), 
$$u_{\eps_n}\ge \frac{t}{2}\quad\text{on}\quad D_\delta\cap \{u_0>t\}.$$
Thus, we have 
$$J_\beta(v,E)\ge \beta \int_{D_\delta\cap \{u_0>t\}\cap\partial^\ast\Omega_{\eps_n}}u_{\eps_n}^2\,d\HH^{d-1}\ge \frac{\beta t^2}2\,{\rm Per}\,\big(\Omega_{\eps_n};D_\delta\cap \{u_0>t\}\big).$$
Now, if we choose $t$ such that $Per(\{u_0>t\})<\infty$ (which, by the co-area formula, is true for almost-every $t>0$), then we have that 
$$\text{Per}\big(\Omega_{\eps_n}\cap \{u_0>t\}\cap D_{\delta}\big)\le C_{t,\delta}\quad\text{for every}\quad n\ge N_{t,\delta},$$
for some constant $C_{t,\delta}>0$. Now, since all the sets $\Omega_{\eps_n}\cap \{u_0>t\}\cap D_{\delta}$ are contained in $D$ and have uniformly bounded perimeter, we can find a set $\Omega_0$ and a subsequence for which 
$$\ind_{\Omega_{\eps_n}\cap \{u_0>t\}\cap D_{\delta}}(x)\to \ind_{\Omega_{0}\cap \{u_0>t\}\cap D_{\delta}}(x)\quad\text{for almost-every}\quad x\in D.$$
Thus, by a diagonal sequence argument,  we can extract a subsequence of $\eps_n$ (still denoted by $\eps_n$) and we can define the set $\Omega_0\subset\R^d$ as the pointwise limit  
$$\ind_{\Omega_0}(x)=\lim_{n\to\infty}\ind_{\Omega_{\eps_n}\cap\{u_0>0\}}(x)\quad\text{for almost-every}\quad x\in \{u_0>0\},$$
and we notice that, by construction, $\Omega_0\subset\{u_0>0\}$.
Notice that, we do not know a priori that $\Omega_0$ has finite perimeter. We only know that 
$$\text{\rm Per}\,(\Omega_0\cap\{u_0>t\}\cap D_\delta)<\infty\quad\text{for every}\quad\delta>0\quad\text{and almost-every}\quad t>0.$$
which means that $\Omega_0\cap\{u_0>t\}$ has locally finite perimeter in $D$ for almost-every $t>0$.
\subsection{An optimality condition}
As pointed out above, we do not know if the couple $(u_0,\Omega_0)$ is even an admissible competitor for \eqref{e:pb} (we need to show that $\Omega_0\in \mathcal E$). Nevertheless, we can still prove that it satisfies a suitable optimality condition.
\begin{lemma}[The optimality condition at the limit]\label{l:opt_cond}
Let $u_0$ and $\Omega_0$ be as in Section \ref{sub:construction}. Then, for almost-every $t>0$, we have 
\begin{equation}\label{e:opt_cond}
\int_{\{u_0<t\}}|\nabla u_0|^2\,dx\le \beta\, t^2\, \text{\rm Per}\big(\{u_0<t\}\big).
\end{equation}
\end{lemma}	
\begin{proof}
Let now $t>0$ be fixed and such that the set $\{u_0<t\}$ has finite perimeter. 	Then, for $n$ large enough, we can use the couple 
$(u_0\vee t, \Omega_0\cup\{u_0< t\})$ to test the optimality of $(u_{\eps_n},\Omega_{\eps_n})$.
Notice that the set $\Omega_0\cup\{u_0< t\}$ has finite perimeter for a.e $t\in (0,m)$, as observed in the previous section.
 For the sake of simplicity, we write $u_{\eps_n}=u_n$, $\Omega_{\eps_n}=\Omega_n$, $u_0=u$ and $\Omega_0=\Omega$. Thus, we have 
\begin{align}
\int_{D}|\nabla u_{n}|^2\,dx+&\beta\int_{\{u> t\}\cap \partial^\ast\Omega_{n}}u_{n}^2\,d\HH^{d-1}\notag\\
&\le \int_{D}|\nabla u_{n}|^2\,dx+\beta\int_{\partial^\ast\Omega_{n}}u_{n}^2\,d\HH^{d-1}\notag\\
&\le \int_{D}|\nabla (u\vee t)|^2\,dx+\beta\int_{\partial^\ast(\Omega\cup\{u<t\})}u^2\,d\HH^{d-1}\notag\\
&\le \int_{D}|\nabla (u\vee t)|^2\,dx+\beta t^2\, \text{\rm Per}(\{u<t\})+\beta\int_{\{u> t\}\cap \partial^\ast\Omega}u^2\,d\HH^{d-1}.\label{e:att}
\end{align}
Now, by the weak convergence of $u_n$ to $u$, we get that 
$$\int_{D}|\nabla u|^2\,dx\le \liminf_{n\to\infty}\int_{D}|\nabla u_{n}|^2\,dx\,.$$
On the other hand, setting $U_{t,\delta}$ to be the open set
$$U_{t,\delta}=\R^d\setminus \big(\overline D_\delta\cap\{u\le t\}\big),$$
for some fixed $\delta>0$, and applying Lemma \ref{l:semicontinuity}, we have that
\begin{align*}
 \int_{U_{t,\delta}\cap \partial^\ast\Omega}u^2\,d\HH^{d-1}&\le  \liminf_{n\to\infty}\int_{U_{t,\delta}\cap \partial^\ast\Omega_{n}}u_{n}^2\,d\HH^{d-1}\le  \liminf_{n\to\infty}\int_{\{u>t\}\cap \partial^\ast\Omega_{n}}u_{n}^2\,d\HH^{d-1}.
\end{align*}
Taking the limit as $\delta\to0$, by the monotone convergence theorem, we get that
$$\lim_{\delta\to0}\int_{U_{t,\delta}\cap \partial^\ast\Omega}u^2\,d\HH^{d-1}=\int_{\big(\R^d\setminus (D\cap\{u\le t\})\big)\cap \partial^\ast\Omega}u^2\,d\HH^{d-1}$$
Now since 
$$u(x)=h(x)\quad\text{for quasi-every}\quad x\in \R^d\setminus D\quad\text{and for}\quad \HH^{d-1}\text{-almost-every}\quad x\in \R^d\setminus D,$$ 
and since $h\ge m>t$ on $\partial D$, we have that 
\begin{equation}\label{e:one_more_light2}
\int_{\big(\R^d\setminus (D\cap\{u\le t\})\big)\cap \partial^\ast\Omega}u^2\,d\HH^{d-1}=\int_{\{u>t\}\cap \partial^\ast\Omega}u^2\,d\HH^{d-1}.
\end{equation}
Thus, we get that
\begin{equation}\label{e:one_more_light}
\int_{\{u>t\}\cap \partial^\ast\Omega}u^2\,d\HH^{d-1}\le \liminf_{n\to\infty}\int_{D\cap\{u>t\}\cap \partial^\ast\Omega_{n}}u_{n}^2\,d\HH^{d-1}.
\end{equation}
Now, using \eqref{e:one_more_light} and \eqref{e:att}, we obtain
\begin{align}
\int_{D}|\nabla u|^2\,dx+&\beta\int_{\{u> t\}\cap \partial^\ast\Omega}u^2\,d\HH^{d-1}\notag\\
&\le \liminf_{n\to\infty}\int_{D}|\nabla u_{n}|^2\,dx+\beta\int_{\{u> t\}\cap \partial^\ast\Omega_{n}}u_{n}^2\,d\HH^{d-1}\notag\\
&\le \int_{D}|\nabla (u\vee t)|^2\,dx+\beta t^2\, \text{\rm Per}(\{u<t\})+\beta\int_{\{u> t\}\cap \partial^\ast\Omega}u^2\,d\HH^{d-1},\notag
\end{align}
which gives \eqref{e:opt_cond}.
\end{proof}	

\subsection{Non-degeneracy} The crucial observation in this section is that the functions $u$ satisfying the optimality condition \eqref{e:opt_cond} are non-degenerate in the sense of the following proposition.
\begin{prop}[Non-degeneracy]\label{p:non-degeneracy}
Let $\beta>0$, $m>0$, $D$ be a bounded open set of $\R^d$ and $u\in H^1(D)$ be a non-negative function in $D$ such that $u\ge m$ on $\partial D$. Let $\Omega\subset D$ be a set of finite perimeter in $D$. Suppose that $u$ and $\Omega$ satisfy the optimality condition 
\begin{equation}\label{e:opt_cond_nd}
\int_{\Omega_t}|\nabla u|^2\,dx\le \beta\, t^2\, \text{\rm Per}(\Omega_t)\qquad\text{where}\qquad\Omega_t=\{u\le t\},
\end{equation}
for almost-every $t\in(0,m)$. Then, $|\Omega_t|=0$ for some $t>0$.
\end{prop}	
\begin{proof}
By contradiction, suppose that 
$$|\Omega_t|>0\quad\text{for every}\quad t>0.$$
Let $t\in(0,m)$ be fixed. By the co-area formula, the Cauchy-Schwartz inequality and the optimality condition \eqref{e:opt_cond_nd}, we get  
\begin{equation}\label{e:main_inequality}
\int_{\Omega_t}|\nabla u|=\int_0^t\text{\rm Per}(\Omega_s)\,ds\le \left(\int_{\Omega_t}|\nabla u|^2\right)^{\sfrac12}|\Omega_t|^{\sfrac12}\le t\beta^{\sfrac12}\text{\rm Per}(\Omega_t)^{\sfrac12}|\Omega_t|^{\sfrac12}.
\end{equation}
We now set 
$$f(t):=\int_{0}^t \text{\rm Per}(\Omega_s)\,ds=\int_{\Omega_t}|\nabla u|\,dx\ .$$
Using \eqref{e:main_inequality}, we will estimate $f(t)$ from below. 
\medskip

\noindent{\bf Step 1. \it Non-degeneracy of $f$.}
By the isoperimetric inequality and the estimate \eqref{e:main_inequality}, there is a dimensional constant $C_d$ such that
$$\int_{0}^t \text{\rm Per}(\Omega_s)\,ds\le t\beta^{\sfrac12}C_d\,\text{\rm Per}(\Omega_t)^{\frac{2d-1}{2d-2}}\ .$$
Using the definition of $f$, we can re-write this inequality as
$$f(t)^{\frac{2d-2}{2d-1}}\le t^{\frac{2d-2}{2d-1}}\big(\beta^{\sfrac12}C_d\big)^{\frac{2d-2}{2d-1}}f'(t)\ .$$
After rearranging the terms and integrating from $0$ to $t$, we obtain 
$$f(t)^{\frac1{2d-1}}-f(0)^{\frac1{2d-1}}\ge \frac{t^{\frac1{2d-1}}}{\big(\beta^{\sfrac12}C_d\big)^{\frac{2d-2}{2d-1}}}\ .$$
Now, since $u$ is non-negative in $D$, we have that $f(0)=0$. Thus 
$$f(t)\ge \frac{t}{\big(\beta^{\sfrac12}C_d\big)^{{2d-2}}}\ .$$ 
Setting 
\begin{equation}\label{e:def_C}
C=\big(\beta C_d\big)^{1-d},
\end{equation}
we obtain the lower bound
$$f(t)\ge Ct.$$
In particular, as a consequence of \eqref{e:main_inequality}, we get that 
\begin{equation}\label{e:induction0}
C\le \beta^{\sfrac12} \text{\rm Per}(\Omega_t)^{\sfrac12}|\Omega_t|^{\sfrac12}.
\end{equation}
\noindent{\bf Step 2. \it Non-degeneracy of $|\Omega_t|$.}
Let $\alpha\in(0,1)$ be fixed. Then, we have that 
\begin{align*}
\int_0^t \text{\rm Per}(\Omega_s)^\alpha\,|\Omega_s|^{1-\alpha} ds&\le \left(\int_0^t \text{\rm Per}(\Omega_s)\,ds\right)^{\alpha}\left(\int_0^t|\Omega_s|\,ds\right)^{1-\alpha}\\
&\le \left(t\beta^{\sfrac12} \text{\rm Per}(\Omega_t)^{\sfrac12}|\Omega_t|^{\sfrac12}\right)^{\alpha}\Big(t|\Omega_t|\Big)^{1-\alpha}=t\beta^{\sfrac{\alpha}{2}} \text{\rm Per}(\Omega_t)^{\sfrac\alpha2}|\Omega_t|^{1-\sfrac\alpha2}.
\end{align*}
Thus, we obtain that for fixed $T\in(0,m)$ and $C>0$, the following implication holds :

\begin{equation}\label{e:claim_non_deg}
\begin{array}{ll}
\text{If $\quad C\le \text{\rm Per}(\Omega_t)^\alpha|\Omega_t|^{1-\alpha}\quad $for every$\quad t\in(0,T)$,}\medskip\\
\qquad\qquad\qquad \text{then$\quad C\le \beta^{\sfrac\alpha2} \text{\rm Per}(\Omega_t)^{\sfrac\alpha2}|\Omega_t|^{1-\sfrac\alpha2}\quad $for every$\quad t\in(0,T)$.}
\end{array}
\end{equation}
%
%
%
%

\noindent \rm We claim that, for every $n\ge 1$ and every $t\in(0,m)$, we have the inequality 
\begin{equation}\label{e:induction}
C\le \beta^{1-\sfrac{1}{2^{n}}}\text{\rm Per}(\Omega_t)^{\sfrac{1}{2^{n}}}|\Omega_t|^{1-\sfrac{1}{2^{n}}}.
\end{equation}
In order to prove \eqref{e:induction}, we argue by induction on $n$. When $n=1$, \eqref{e:induction} is precisely \eqref{e:induction0}. In order to prove that the claim \eqref{e:induction} for $n\in \N$ implies the same claim for $n+1$, we apply \eqref{e:claim_non_deg} for $\alpha=2^{-n}$, $n\in\N$, which gives precisely \eqref{e:induction} with $n+1$. This concludes the proof of \eqref{e:induction}. Next, passing to the limit as $n\to\infty$, we obtain that
$$C\le \beta |\Omega_t|\quad\text{for every}\quad t\in(0,T),$$
where $C$ is given by \eqref{e:def_C}. Thus, there is a dimensional constant $C_d>0$ such that 
\begin{equation}\label{e:measure_non_degeneracy}
\beta^{-d} C_d\le |\Omega_t|\qquad\text{for every}\qquad t\in[0,m).
\end{equation}
\noindent{\bf Step 3. \it Conclusion.} We now notice that 
$$\lim_{t\to0}|\Omega_t|=|\Omega_0|>0.$$
Thus, for every $\eps>0$, there is $T_\eps$ such that for all $t\in(0,T_\eps)$ we have  
\begin{equation}
\int_{\Omega_t}|\nabla u|=\int_0^t\text{\rm Per}(\Omega_s)\,ds\le \left(\int_{\Omega_t}|\nabla u|^2\right)^{\sfrac12}|\Omega_t\setminus\Omega_0|^{\sfrac12}\le t\eps ^{\sfrac12}\text{\rm Per}(\Omega_t)^{\sfrac12}|\Omega_t|^{\sfrac12}.
\end{equation}
Now, repeating the argument fro Step 1 and Step 2, we get that \eqref{e:measure_non_degeneracy} should hold with $\eps$ in place of $\beta$. Since $\eps>0$ is arbitrary, this is a contradiction.
\end{proof}

\subsection{Existence of a solution}
We are now in position to prove that the couple $(u_0,\Omega_0)$, constructed in Section \ref{sub:construction}, is a solution to \eqref{e:pb}.

\begin{prop}[Existence of a solution]
There is a dimensional constant $C_d>0$ such that if $D$ is a bounded open set of $\R^d$ and $\beta>0$ is a given posiotive constant, then the following holds. For every set $E\subset\R^d$ of finite perimeter and every $v\in H^1(\R^d)$ satisfying
$$v\ge m\quad\text{on}\quad D\quad\text{for some constant}\quad m>0,$$
there is a solution $(u,\Omega)$ of the problem \eqref{e:pb}.
\end{prop}	
\begin{proof}
Let $(u_0,\Omega_0)$ be as in Section \ref{sub:construction}. Then, by Lemma \ref{l:opt_cond}, $(u_0,\Omega_0)$ satisfies the optimality condition \eqref{e:opt_cond_nd}.  Now, by Proposition \ref{p:non-degeneracy} we get that $u_0\ge t$ in $D$, for some $t>0$. In particular, $\Omega_0$ has finite perimeter in $D$. Precisely, for every $\delta>0$, we have 
\begin{align*}
\text{Per}(\Omega_0;D_\delta)&\le \liminf_{n\to\infty}\text{Per}(\Omega_{\eps_n};D_\delta)\le \frac4{t^2}\liminf_{n\to\infty}\int_{D_\delta\cap\partial^\ast\Omega_{\eps_n}}u_{\eps_n}^2\,d\HH^{d-1}\\
&\le \frac4{\beta t^2}\liminf_{n\to\infty}J_\beta\big(u_{\eps_n},\Omega_{\eps_n}\big)\le \frac4{\beta t^2} J_\beta(v,E).
\end{align*}
Passing to the limit as $\delta\to0$, we get 
\begin{align*}
\text{Per}(\Omega_0;D)\le \frac4{\beta t^2} J_\beta(v,E).
\end{align*}
In particular, this implies that $\Omega_0$ is a set of finite perimeter in $\R^d$. Indeed, 
\begin{align*}
\text{Per}(\Omega_0)&\le \text{Per}(\Omega_0;D)+2\text{Per}(D)+\text{Per}(\Omega_0;\R^d\setminus \overline D)\\
&\le \frac4{\beta t^2} J_\beta(v,E)+2\text{Per}(D)+\text{Per}(E;\R^d\setminus \overline D).
\end{align*}
Thus, the couple $(u_0,\Omega_0)$ is admissible in \eqref{e:pb}; it now remains to prove that it is optimal. Let $\widetilde u\in H^1(D)$ be non-negative on $D$ and such that $u-v\in H^1_0(D)$. Let $\widetilde \Omega\subset\R^d$ be a set of finite perimeter such that $\widetilde \Omega=E$ on $\R^d\setminus D$. It is sufficient to prove that 
$$J_\beta(u_0,\Omega_0)\le J_\beta(\widetilde u,\widetilde\Omega).$$
Let $\eps>0$ be fixed. We now use the couple $(\widetilde u\vee \eps,\widetilde\Omega)$ to test the optimality of $\big(u_{\eps_n},\Omega_{\eps_n}\big)$ :
$$J_\beta\big(u_{\eps_n},\Omega_{\eps_n}\big)\le  J_\beta(\widetilde u\vee \eps,\widetilde\Omega).$$
Passing to the limit as $\eps\to0$, we get 
$$J_\beta\big(u_{\eps_n},\Omega_{\eps_n}\big)\le  J_\beta(\widetilde u,\widetilde\Omega).$$
Now, Lemma \ref{l:semicontinuity} and the semicontinuity of the $H^1$ norm gives that $J_\beta (u_0,\Omega_{0})\le  J_\beta(\widetilde u,\widetilde\Omega)$,
which concludes the proof.
\end{proof}
	
\section{Regularity of the free boundary}\label{s:regularity}
In this section, we prove the regularity of the free boundary. In Theorem \ref{t:regularity}, we prove that the solutions of \eqref{e:pb} are almost-minimizers for the perimeter in $D$. As a consequence, $\partial\Omega$ can be decomposed into a regular and a singular part and that the regular part is $C^{1,\alpha}$ manifold. Then, in Theorem \ref{t:higher_regularity}, we prove that the regular part of the free boundary is $C^\infty$ smooth. 
\begin{teo}\label{t:regularity}
Let $(u,\Omega)$ be a solution to \eqref{e:pb}. 
there is a constant $C>0$ such that $\Omega$ is an almost-minimizer of the perimeter  in the following sense: 
$$\text{Per}\,\big(\Omega\,; B_r(x_0)\big) \leq \big(1 + Cr^{\sfrac13}  \big) \text{Per}\,\big(\Omega'; B_r(x_0)\big),$$
for every ball $B_r(x_0)\subset D$ and every set $\Omega'\subset \R^d$ such that $\Omega=\Omega'$ outside $B_r(x_0)$. 

In particular, the free boundary $\partial\Omega\cap D$ can be decomposed as the disjoint union of a regular part $\text{\rm Reg}(\partial\Omega)$ and a singular part $\text{\rm Sing}(\partial\Omega)$, where 
\begin{enumerate}[\quad \rm(i)]
	\item $\text{\rm Reg}(\partial\Omega)$ is a relatively open subset of $\partial\Omega$ and is a $C^{1,\alpha}$ smooth manifold;
	\item $\text{\rm Sing}(\partial\Omega)$ is a closed set, which is empty if $d\le 7$, discrete if $d=8$, and of Hausdorff dimension $d-8$, if $d>8$.
\end{enumerate}	
\end{teo}	
\begin{proof}
We first notice that by Lemma \ref{l:holder}, $u \in C^{0,\sfrac13}(D)$. Let $\delta>0$, $x_0\in D_\delta$ and $r<\sfrac{\delta}2$.\\
We consider a set $\Omega'\subset \R^d$ such that $\Omega'\Delta \Omega \subset\subset B_r(x_0)$. Testing the optimality of $(u,\Omega)$ against $(u,\Omega')$ we get that
		\begin{align*}
		\int_{B_r(x_0)\cap \partial^\ast \Omega}u^2d\HH^{n-1} \leq \int_{B_r(x_0)\cap \partial^\ast \Omega'}u^2d\HH^{n-1},
		\end{align*}
		which implies that
		$$\Big(\min_{B_r(x_0)} u^2\Big)\, \text{Per}\big(\Omega\,; B_r(x_0)\big) \leq  \Big(\max_{B_r(x_0)} u^2\Big)\, \text{Per}\big(\Omega'; B_r(x_0)\big).$$
		By regularity of $u$, we have that 
		$$\ds\max_{B_r(x_0)} u^2 \leq  \min_{B_r(x_0)} u^2 + Cr^{\sfrac13}\le \Big(\min_{B_r(x_0)} u^2\Big)\left(1+\frac{C}{t}r^{\sfrac13}\right),$$
		where in the second inequality, we used that $u\ge t>0$. Thus, we obtain  
		$$\text{Per}\big(\Omega\,; B_r(x_0)\big) \leq \left(1 + \frac{C}{t}r^{\sfrac13}  \right) \text{Per}\big(\Omega'; B_r(x_0)\big),$$
		which proves that $\Omega$ is an almost-minimizer of the perimeter in $D$.
\end{proof}


We next prove that regular part the free boundary $\text{\rm Reg}(\partial\Omega)$ is $C^\infty$. 
\begin{teo}\label{t:higher_regularity}
	Let $(u,\Omega)$ be a solution to \eqref{e:pb}. Let 
	$$D\cap\partial\Omega=\text{\rm Reg}(\partial\Omega)\,\cup\,\text{\rm Sing}(\partial\Omega)$$
	be the decomposition of the free boundary given by Theorem \ref{t:regularity}. Then, in a neighborhood of any point $x_0\in \text{\rm Reg}(\partial\Omega)$,  $\partial\Omega$ is $C^{\infty}$-regular and the function $u$ is $C^\infty$ on $\partial\Omega$.
\end{teo}	
\begin{proof}
We fix a point $x_0\in \text{\rm Reg}(\partial\Omega)$. Without loss of generality, we can assume that $x_0=0$. 

\noindent{\bf Step 1.\,\it Notation.} For any $x\in \R^d$, we use the notation $x=(x',x_d)$, where $x'\in \R^{d-1}$ and $x_d\in\R$.  By the $C^{1,\alpha}$ regularity of $\text{\rm Reg}(\partial\Omega)$, in a neighborhood of the origin $B'\times(-\eps,\eps)\subset\R^{d-1}\times\R$, $\partial\Omega$ is the graph of a $C^{1,\alpha}$ regular function 
$\eta:B'\to\R$, where $B'$ is a ball in $\R^{d-1}$; the set $\Omega$ coincides with the subgraph of $\eta$ in a neighborhood of the origin:
$$B'\times(-\eps,\eps)\cap \Omega=\big\{(x',x_d)\in B'\times (-\eps;\eps)\ :\ x_d<\eta(x')\big\}.$$
and the exterior normal $\nu_\Omega$ is given by 
\begin{equation}\label{e:nu_Omega}
\nu_\Omega=\frac{(-\nabla_{\!x'}\eta,1)}{\sqrt{1+|\nabla_{\!x'}\eta|^2}},
\end{equation}
where $\nabla_{\!x'}\eta$ is the gradient of $\eta$ in the first $d-1$ variables.  Let $u_+$ and $u_-$ be the restrictions of $u$ on the sets $\overline\Omega$ and $D\setminus \Omega$; since $u$ is continuous across $\partial\Omega$, we have $u_+=u_-$ on $\partial\Omega$. Moreover, we write the gradients of $u_+$ and $u_-$ as 
$$\nabla u_\pm=\big(\nabla_{\!x'}u_{\pm},\partial_{x_d}u_{\pm}\big)\in\R^{d-1}\times\R.$$

\noindent{\bf Step 2.\,\it Transmission condition and $C^{1,\alpha}$ regularity of $u$.} In Lemma \ref{l:robin_condition}, we keep fixed and free boundary $\partial\Omega$ and we use vertical perturbations of the function $u$ to obtain a Robin-type transmission condition on $\partial\Omega$. We notice that the recent results \cite{css} and \cite{D} imply the $C^{1,\alpha}$-regularity of $u_+$ and $u_-$, up to the boundary $\partial\Omega$. Thus, the gradient is well-defined and the transmission conditions \eqref{e:robin0} hold in the classical sense. 

\noindent{\bf Step 3.\,\it Optimality condition and $C^{2,\alpha}$ regularity of $ \text{\rm Reg}(\partial\Omega)$.} In Lemma \ref{l:reg2} we perform variatons of the optimal set to find the geometric equation solved by $\partial\Omega$. Precisely, we find that the curvature of the optimal set solves an equation of the form 
$$\text{``Mean curvature of $\partial\Omega$"}=F(\nabla u_+,\nabla u_-,u_\pm)\quad\text{on}\quad \partial\Omega.$$
In particular, this implies that if $u$ is $C^{k,\alpha}$, for some $k\ge 1$, then $\partial\Omega$ is $C^{k+1,\alpha}$.

\noindent{\bf Step 4.\,\it Bootstrap and $C^\infty$ regularity of $\partial\Omega$.} In Lemma \ref{l:reg1} we use the recent results of \cite{D} to show that if the boundary $\partial\Omega$ is $C^{k,\alpha}$ for some $k\ge 2$, then the solutions $u_+$ and $u_-$ are also $C^{k,\alpha}$ regular up to the boundary $\partial\Omega$. Finally, applying this result (Lemma \ref{l:reg1}) and the result from the previous step (Lemma \ref{l:reg2}), we get that $\partial\Omega$ is $C^\infty$. 
\end{proof}

\begin{lemma}[Robin and continuity conditions on $\partial\Omega$]\label{l:robin_condition}
Suppose that $\partial\Omega$ is $C^{1,\alpha}$ regular in the neighborhood of the origin. Let $\eta:B'\to\R$, $u_+$ and $u_-$ be as above. Then, for every $x'\in B'$ we have 
\begin{equation}\label{e:robin0}
\begin{cases}
\begin{array}{ll}
\nabla_{\!x'}\eta\cdot \nabla_{\!x'}u_+-\nabla_{\!x'}\eta\cdot \nabla_{\!x'}u_-=-\big(\partial_{x_d}u_+-\partial_{x_d}u_-\big)|\nabla_{\!x'}\eta|^2\\
\sqrt{1+|\nabla_{\!x'}\eta|^2}\big(\partial_{x_d}u_+-\partial_{x_d}u_-\big)+\beta u=0,
\end{array}
\end{cases}
\end{equation}	
where $u_+$, $u_-$ and their partial derivatives are calculated in $(x',\eta(x'))\in \partial\Omega$. 
\end{lemma}	
\begin{proof}
Let $\phi\in C^{\infty}_c(D)$ be a smooth function supported in $B'\times(-\eps,\eps)$.
Then, the optimality of $u$ gives that
\begin{align*}
0=\frac{\partial}{\partial t}\Big\vert_{t=0}J_\beta(u+t\phi,\Omega)&=\int_{D\setminus\partial\Omega}2\nabla u\cdot\nabla \phi\,dx+\beta\int_{\partial\Omega}2u\phi\,d\HH^{d-1}\\
&=\int_{\partial\Omega}2\big(\nu_\Omega\cdot\nabla u_+-\nu_\Omega\cdot\nabla u_-+\beta u\big)\phi\,d\HH^{d-1},
\end{align*}
where in the last inequality we integrated by parts $u_+$ in $\Omega$ and $u_-$ in $D\setminus\Omega$. 
Since $\phi$ is arbitrary we get that $u$ satisfies the Robin-type condition on $\partial\Omega$
\begin{equation}\label{e:robin}
\nu_\Omega\cdot\nabla u_+-\nu_\Omega\cdot\nabla u_-+\beta u\qquad\text{on}\qquad \partial\Omega.
\end{equation}
Now, using \eqref{e:nu_Omega}, we can re-write this as 
\begin{equation}\label{e:robin2}
\big(-\nabla_{\!x'}\eta\cdot\nabla_{\!x'}u_++\partial_{x_d}u_+\big)-\big(-\nabla_{\!x'}\eta\cdot\nabla_{\!x'}u_-+\partial_{x_d}u_-\big)+\beta u \sqrt{1+|\nabla_{\!x'}\eta|^2}=0.
\end{equation}
On the other hand $u$ is continuous across $\partial\Omega$. This means that
$$\nabla_{\!x'}u_+(x',\eta(x'))+\partial_{x_d}u_+(x',\eta(x'))\nabla_{\!x'}\eta=\nabla_{\!x'}u_-(x',\eta(x'))+\partial_{x_d}u_-(x',\eta(x'))\nabla_{\!x'}\eta.$$
Multiplying by $\nabla_{\!x'}\eta$, we get 
\begin{equation}\label{e:robin3}
\nabla_{\!x'}\eta\cdot \nabla_{\!x'}u_++\partial_{x_d}u_+|\nabla_{\!x'}\eta|^2=\nabla_{\!x'}\eta\cdot \nabla_{\!x'}u_-+\partial_{x_d}u_-|\nabla_{\!x'}\eta|^2,
\end{equation}
where $u_+$, $u_-$ and their partial derivatives are calculated in $(x',\eta(x'))$. Putting together \eqref{e:robin2} and \eqref{e:robin3}, we get \eqref{e:robin0}. 
\end{proof}

\begin{lemma}[Smooth boundary $\Rightarrow$ smooth function]\label{l:reg1}
Let $(u,\Omega)$ be a solution of \eqref{e:pb}. Suppose that, in a neighborhood of zero, $\partial\Omega$ is $C^{k,\alpha}$-regular for some $k\ge 1$. Then, in a neighborhood of the origin, the functions $u_+$ and $u_-$ are $C^{k,\alpha}$ up to the boundary $\partial\Omega$.
\end{lemma}	
\begin{proof}
We argue by induction. The case $k=1$ follows by \cite{D}. We suppose that $k\ge 2$ and that the claim holds for $k-1$. 
Suppose that $\partial\Omega$ is the graph of $\eta:B'\to\R$, $\eta\in C^{k,\alpha}(B')$, and consider the functions
$$v_+(x',x_d):=u_+(x',x_d+\eta(x'))\qquad\text{and}\qquad v_-(x',x_d):=u_-(x',x_d+\eta(x')),$$
defined on the half-space $\{x_d\ge 0\}$.
We set 
$$A_\eta=\begin{pmatrix}
\mathcal N_{d-1} & -(\nabla_{\!x'}\eta)^t\\
-\nabla_{\!x'}\eta & |\nabla_{\!x'}\eta|^2
\end{pmatrix},$$
where $\mathcal N_{d-1} $ is the null $(d-1)\times (d-1)$ matrix and we notice that $A_\eta$ has $C^{k-1,\alpha}$ regular coefficients.
Now, since $u_+$ and $u_-$ are harmonic in $\Omega$ and $D\setminus\overline\Omega$,  we have that $v_+$ and $v_-$ are solutions to the transmission problem 
$$\begin{cases}
\begin{array}{rl}
-\text{div}((Id+A_\eta)\nabla v_+)=0&\quad\text{in}\quad\{x_d>0\}\\
-\text{div}((Id+A_\eta)\nabla v_-)=0&\quad\text{in}\quad\{x_d<0\}\\
v_+=v_-&\quad\text{on}\quad\{x_d=0\}\\
\ds\partial_{x_d}v_+-\partial_{x_d}v_-+\frac{\beta}{2\sqrt{1+|\nabla_{\!x'}\eta|^2}}(v_++v_-)=0&\quad\text{on}\quad\{x_d=0\}.
\end{array}
\end{cases}$$
We now fix $k-1$ directions $i_1,\dots,i_{k-1}$, $i_j\neq d$ for every $j$, and we consider the functions $$w_+:=\partial_{i_1}\partial_{i_2}\dots\partial_{i_{k-1}}v_+\qquad\text{and}\qquad  w_-:=\partial_{i_1}\partial_{i_2}\dots\partial_{i_{k-1}}v_-.$$
We notice that, in $\{x_d>0\}$ and $\{x_{d}<0\}$ the functions $w_+$ and $w_-$ are solutions to 
$$
\ds-\text{div}\big((Id+A_\eta)\nabla w_\pm\big)+\sum_{I,J}\text{div}\big(\partial_IA_\eta\partial_J\nabla u_\pm\big)=0,$$
where the sum is over all multiindices $I$ and $J$ such that the sets $I$ and $J$ are disjoint subsets of $\{i_1,i_2,\dots,i_{k-1}\}$, $I\cup J=\{i_1,i_2,\dots,i_{k-1}\}$ and  $I$ is non-empty. In particular, using that $A_\eta\in C^{k-1,\alpha}$ and $\nabla u\in C^{k-2,\alpha}$ (since by hypothesis $u_\pm\in C^{k-1,\alpha}$), we get that $w_\pm$ solve 
$$\ds-\text{div}\big((Id+A_\eta\big)\nabla w_\pm)+\text{div}(F_\pm)=0\qquad\text{in}\qquad \{\pm x_d>0\},$$
where $F_+$ and $F_-$ are $C^{0,\alpha}$ continuous functions (depending on $i_1,\dots,i_k$). On the other hand, on the boundary $\{x_d=0\}$ we have that $w_+=w_-$ and 
$$\ds\partial_{x_d}w_+-\partial_{x_d}w_-+\partial_{i_1}\partial_{i_2}\dots\partial_{i_{k-1}}\left(\frac{\beta (u_++u_-)}{2\sqrt{1+|\nabla_{x'}\eta|^2}}\right)=0\quad\text{on}\quad\{x_d=0\}.$$
Reasoning as above, we notice that this condition can be written as 
$$\ds\partial_{x_d}w_+-\partial_{x_d}w_-=g\quad\text{on}\quad\{x_d=0\},$$
where $g$ is a $C^{0,\alpha}$ function. 
Now, applying \cite[Theorem 1.2]{D}, we get that $w_+$ and $w_-$ are $C^{1,\alpha}$ regular up to the boundary $\{x_d=0\}$. 
Thus, the trace $u_+=u_-$ is $C^{k,\alpha}$ smooth on $\{x_d=0\}$. Finally, the classical Schauder estimates give that $u_+$ and $u_-$ are $C^{k,\alpha}$ on $\{x_d\ge 0\}$ and $\{x_d\le 0\}$, respectively. 
\end{proof}

\begin{lemma}[Smooth function $\Rightarrow$ smooth boundary] \label{l:reg2}
Let $(u,\Omega)$ be a solution of \eqref{e:pb}. Suppose that, in a neighborhood of zero, $\partial\Omega$ is $C^{1,\alpha}$-regular and that the functions $u_+$ and $u_-$ are $C^{k,\alpha}$ up to the boundary $\partial\Omega$, for some $k\ge 1$. Then, 	$\partial\Omega$ is $C^{k+1,\alpha}$-regular in a neighborhood of zero.
\end{lemma}	
\begin{proof}
Let $\xi\in C^\infty_c(D ;\R^d)$ be a given vector field with compact support in $D$ and let $\Psi_t$ be the function
$$\Psi_t(x)=x+t\xi (x)\quad\text{for every}\quad x\in D.$$ 
Then, for $t$ small enough, $\Psi_t: D\to D$ is a diffeomorphism and setting $\Phi_t:=\Psi_t^{-1}$, the function $u_t:=u\circ \Phi_t$ is well-defined and belongs to $H^1(D)$; the function $\ds t\mapsto\int_{D}|\nabla u_t|^2\,dx$ is differentiable at $t=0$ and 
$$\frac{\partial}{\partial t}\Big\vert_{t=0}\int_{D}|\nabla u_t|^2\,dx= \int_{D}\left(-2\nabla u\, D\xi\cdot \nabla u+ |\nabla u|^2\text{\rm div}\,\xi\,\right)dx\,.$$
It is immediate to check that 
$$-2\nabla u\, D\xi\cdot \nabla u+ |\nabla u|^2\text{\rm div}\,\xi = \text{div}\big(|\nabla u|^2\xi-2(\xi\cdot\nabla u)\nabla u\big)\qquad\text{in}\qquad D\setminus\partial\Omega\,.$$
We now take $\xi$ to be smooth outside $\partial \Omega$ and such that 
$$\xi=\phi \nu_\Omega\quad\text{on}\quad \partial\Omega,$$ 
where $\nu_\Omega$ is the exterior normal to 
$\partial\Omega$ and $\phi:\partial\Omega\to\R$ is continuous and with compact support. Integrating by parts, we get 
\begin{align*}
\frac{\partial}{\partial t}\Big\vert_{t=0}\int_{D}|\nabla u_t|^2\,dx&=\int_{\partial\Omega}\Big(|\nabla u_+|^2(\xi\cdot\nu_\Omega)-2(\xi\cdot\nabla u_+)(\nu_\Omega\cdot\nabla u_+)\Big)\,d\HH^{d-1}\\
&\qquad-\int_{\partial\Omega}\Big(|\nabla u_-|^2(\xi\cdot\nu_\Omega)-2(\xi\cdot\nabla u_-)(\nu_\Omega\cdot\nabla u_-)\Big)\,d\HH^{d-1},
\end{align*}
where $u_+:=u$ on $\Omega$, and $u_-:=u$ on $D\setminus\overline \Omega$. Now, if 
$$\xi=\phi e_d\qquad\text{and}\qquad\nu_\Omega=\frac{(-\nabla_{\!x'}\eta,1)}{\sqrt{1+|\nabla_{\!x'}\eta|^2}},$$
 then 
\begin{align*}
\sqrt{1+|\nabla_{\!x'}\eta|^2}&\Big(|\nabla u_+|^2(\xi\cdot\nu_\Omega)-2(\xi\cdot\nabla u_+)(\nu_\Omega\cdot\nabla u_+)\Big)\\
&\qquad-\sqrt{1+|\nabla_{\!x'}\eta|^2}\Big(|\nabla u_-|^2(\xi\cdot\nu_\Omega)-2(\xi\cdot\nabla u_-)(\nu_\Omega\cdot\nabla u_-)\Big)\\
&=\phi\Big(|\nabla u_+|^2-|\nabla u_-|^2\Big)\\
&\qquad-2\phi\Big(\,\partial_{x_d}u_+\big(-\nabla_{\!x'}\eta\cdot\nabla_{\!x'}u_++\partial_{x_d}u_+\big)-\partial_{x_d}u_-\big(-\nabla_{\!x'}\eta\cdot\nabla_{\!x'}u_-+\partial_{x_d}u_-\big)\Big).
\end{align*}

We now suppose that $x_0\in \text{Reg}(\partial\Omega)$ and that $\partial \Omega$ is the graph of the ($C^{1,\alpha}$) function $\eta:B'\to\R$, where $B'$ is a ball in $\R^{d-1}$.
Taking $\xi=e_d\phi$ and $\Omega_t=\Phi_t(\Omega)$, we have
\begin{align*}
\frac{\partial}{\partial t}\Big\vert_{t=0}\int_{\partial\Omega_t}u_t^2\,d\HH^{d-1}&=\frac{\partial}{\partial t}\Big\vert_{t=0}\int_{B'}u^2\big(x',\eta(x')\big)\sqrt{1+|\nabla_{\!x'}\eta+t\nabla_{\!x'}\phi|^2}\,dx'\\
&=\int_{B'}\frac{u^2\big(x',\eta(x')\big)}{\sqrt{1+|\nabla_{\!x'}\eta|^2}}\,\nabla_{\!x'}\eta\cdot\nabla_{\!x'}\phi\,dx'\\
&=\int_{B'}u^2\big(x',\eta(x')\big)H\big(x',\eta(x')\big)\phi(x')\,dx'\\
&\qquad -2\int_{B'}\phi(x')u\big(x',\eta(x')\big)\frac{\big(\nabla_{\!x'}u+\partial_{x_d}u\nabla_{\!x'}\eta\big)\cdot\nabla_{\!x'}\eta}{\sqrt{1+|\nabla_{\!x'}\eta|^2}}\,dx'.
\end{align*} 
In particular, combining these two computations and using the optimality of $(u,\Omega)$, we get 
\begin{align*}
0=\frac{\partial}{\partial t}\Big\vert_{t=0}J_\beta(u_t,\Omega_t)&=\int_{B'}\beta u^2 H(x')\phi(x')\,dx'\\
&\qquad +\int_{B'}\Big(|\nabla u_+|^2-|\nabla u_-|^2\Big)\phi(x')\,dx'\\
&\qquad\qquad -\int_{B'}2\left(1+|\nabla_{\!x'}\eta|^2\right)\left((\partial_{x_d} u_+)^2-(\partial_{x_d}u_-)^2\right)\phi(x')\,dx'
\end{align*}
Since $\phi$ is arbitrary, we obtain that $\eta$ is a solution of the problem 
$$-\text{div}_{x'}\left(\frac{\nabla_{\!x'}\eta}{\sqrt{1+|\nabla_{\!x'}\eta|^2}}\right)=f(x')\qquad\text{in}\qquad B',$$
where 
\begin{align*}
f(x')&=\frac1{\beta u^2\big(x',\eta(x')\big)}\Big[\Big(|\nabla u_+|^2-|\nabla u_-|^2\Big)-2\big(1+|\nabla_{\!x'}\eta|^2\big)\Big((\partial_{x_d} u_+)^2-(\partial_{x_d}u_-)^2\Big)\Big],
\end{align*}
and all the derivatives of $u_+$ and $u_-$ are calculated at $\big(x',\eta(x')\big)$. Since the right-hand side $f$ is $C^{k-1,\alpha}$ regular, we get that $\eta$ is $C^{k+1,\alpha}$ regular. This concludes the proof. 
\end{proof}

\appendix

\section{Examples of minimizers}	
In this section, we use a calibration argument to prove that if $E=\{x_d>0\}$ and $v\equiv 1$, then in any Steiner symmetric set $D\subset\R^d$, the solution $(\Omega,u)$ is unique, $u$ is even with respect to the hyperplane $\{x_d=0\}$ and $\Omega$ is precisely the half-space $E$. Our main result is the following.
\begin{prop}
	Let $D$ be an open set, Steiner symmetric with respect to the hyperplane $\{x_d=0\}$. Let $E$ be the half-ball $E=B\cap\{x_d>0\}$, for some large ball $B$ containing $D$,  and let $v\equiv 1$. Then there is a unique solution $(u,\Omega)$ to \eqref{e:pb}, where $\Omega=E$, $u$ is positive and even with respect to $\{x_d=0\}$ and  solves the equation 
	\begin{equation}\label{e:esempio}
	\Delta u=0\quad\text{in}\quad \{x_d>0\}\cap D,\qquad 
	\partial_{x_d} u=\frac{1}2\beta u\quad\text{on}\quad  D\cap\{x_d=0\}.
	\end{equation}
\end{prop}	
\begin{proof}
	Let $\widetilde u\in \mathcal V$ and $\widetilde\Omega\in \mathcal E$ be given. We will prove that 
	$$J_\beta(u,\Omega)\le J_\beta(\widetilde u,\widetilde\Omega),$$
	with an equality, if and only if, $(u,\Omega)=(\widetilde u,\widetilde\Omega)$. First, we notice that, since $J_\beta(1\wedge\widetilde  u\vee0,\widetilde\Omega)\le J_\beta(\widetilde u,\widetilde\Omega)$, we can suppose that $0\le \widetilde u\le 1$. We then write $\widetilde u$ as $\widetilde u=1-\varphi$ for some $\varphi\in H^1_0(D)$ such that $0\le \varphi\le 1$ and we define the function $\widetilde u_\ast=1-\varphi_\ast$, where $\varphi_\ast\in H^1_0(D)$ is the Steiner symmetrization of $\varphi$. We will show that 
	\begin{equation}\label{e:steiner1}
	J_\beta(\widetilde u_\ast,\Omega)\le J_\beta(\widetilde u,\widetilde\Omega)
	\end{equation}
	Indeed, it is well-known that the Steiner symmetrization decreases the Dirichlet energy:
	$$\int_D|\nabla \widetilde u_\ast|^2\,dx=\int_D|\nabla \varphi_\ast|^2\,dx\le \int_D|\nabla \varphi|^2\,dx=\int_D|\nabla \widetilde u|^2\,dx.$$
	In order to estimate also the second term of the energy $J_\beta$, we use a calibration-type argument. 
	We first notice that, by construction, along every line orthogonal to $\{x_d=0\}$, the symmetrized function achieves its maximum in zero. Precisely
	$$\varphi(x',x_d)\le \sup_{x_d}\varphi(x',x_d)=\varphi_\ast(x',0).$$
	Thus, by the definition of $\widetilde u_\ast$, we have 
	\begin{align*}
	\int_{B\cap\partial\widetilde\Omega}\widetilde u^2(x',x_d)\,d\HH^{d-1}&\ge \int_{B\cap\partial\widetilde \Omega}\widetilde u_\ast^2(x',0)\,d\HH^{d-1}\ge \int_{B\cap\partial\widetilde\Omega}\widetilde u_\ast^2(x',0)\,\nu_{\widetilde\Omega}\cdot e_d\,d\HH^{d-1}\\
	&= \int_{B\cap\partial\Omega}\widetilde u_\ast^2(x',0)\, \nu_{\Omega}\cdot(-e_d)\,d\HH^{d-1}+\int_{\Omega\Delta\widetilde\Omega}\text{\rm div}\,(\widetilde u_\ast^2(x',0)e_d)\,dx,
	\end{align*}
	where in order to get the last equality we used the divergence theorem in $\Omega\Delta\widetilde\Omega$. Now, we notice that $\text{\rm div}\,(\widetilde u_\ast^2(x',0)e_d)=0$ and that $\nu_\Omega=-e_d$. 
	Thus, we get 
	$$\int_{B\cap\partial\widetilde\Omega}\widetilde u^2(x',x_d)\,d\HH^{d-1}\ge \int_{B\cap\partial\Omega}\widetilde u_\ast^2\,d\HH^{d-1},$$
	which concludes the proof of \eqref{e:steiner1}. Finally, we notice that the problem 
	\begin{equation*}
	\min\Big\{J_\beta(u,\Omega)\ :\ u\in H^1(D\cap \{x_d>0\}),\quad u=1\ \text{on}\ \partial D\cap\{x_d>0\}\Big\},
	\end{equation*}
	has a unique solution $u$, which is Steiner symmetric, non-negative and satisfies \eqref{e:esempio}.
\end{proof}

\section*{Acknowledgments}
\noindent
The first author was partially supported by  PRIN 2017 {\it Nonlinear
Differential Problems via Variational, Topological and Set-valued Methods} (Grant 2017AYM8XW) and the INdAM-GNAMPA project 2020 ``Problemi di ottimizzazione con vincoli via trasporto ottimo e incertezza''.
The second author was partially supported by the Academy of Finland grant 314227. 
The third author has been partially supported by the European Research Council (ERC) under the European Union's Horizon 2020 research and innovation programme (grant agreement VAREG, No. 853404).

\end{document}